\documentclass[11pt]{amsart}

\usepackage{amsmath,amssymb,amsthm,mathtools}
\usepackage{enumitem}
\usepackage{aliascnt}
\usepackage[colorlinks=true,linkcolor=blue,citecolor=blue,urlcolor=blue]{hyperref}
\usepackage[nameinlink,noabbrev]{cleveref}

\newtheorem{theorem}{Theorem}[section]
\newtheorem{conjecture}{Conjecture}[section]

\newaliascnt{lemma}{theorem}
\newtheorem{lemma}[lemma]{Lemma}
\aliascntresetthe{lemma}

\newaliascnt{corollary}{theorem}

\aliascntresetthe{corollary}

\newaliascnt{observation}{theorem}
\newtheorem{observation}[observation]{Observation}
\aliascntresetthe{observation}

\newaliascnt{proposition}{theorem}
\newtheorem{proposition}[proposition]{Proposition}
\aliascntresetthe{proposition}

\theoremstyle{plain}
\newaliascnt{definition}{theorem}

\aliascntresetthe{definition}

\theoremstyle{remark}
\newaliascnt{remark}{theorem}
\newtheorem{remark}[remark]{Remark}
\aliascntresetthe{remark}

\DeclareMathOperator{\Cay}{Cay}
\DeclareMathOperator{\ex}{ex}
\newcommand{\Z}{\mathbb Z}
\newcommand{\og}{\operatorname{og}}
\newcommand{\floor}[1]{\left\lfloor #1\right\rfloor}

\title[Odd cycles in prime cyclic Cayley graphs]{Odd cycles in symmetric Cayley graphs on prime cyclic groups}

\author[Wei Li]{Wei Li}
\thanks{Corresponding author: Wei Li.}
\address{Fujian Agriculture and Forestry University, Fuzhou, Fujian, China}
\email{liwei@fafu.edu.cn}

\author[Kai Yang]{Kai Yang}
\address{Fuzhou University, Fuzhou, Fujian, China}
\email{443926471@qq.com}

\date{}


\begin{document}

\begin{abstract}
Let $p$ be an odd prime and let $S\subseteq \Z_p$ be symmetric with $0\notin S$.  Let $\Cay(\Z_p,S)$ be the undirected Cayley graph on $\Z_p$ in which $x$ and $y$ are adjacent if and only if $x-y\in S$.  For $1\le \ell\le (p-1)/2$, define
\[
\ex_{\Cay}(C_{2\ell+1},\Z_p)=\max\{|S|: S=-S,\ 0\notin S,\ \Cay(\Z_p,S)\text{ contains no }C_{2\ell+1}\}.
\]
Confirming a conjecture of  Cashman and Kelley, we prove that if $p=2\ell+1$, then $\ex_{\Cay}(C_{2\ell+1},\Z_p)=0$, while if $p>2\ell+1$, then
\[
\ex_{\Cay}(C_{2\ell+1},\Z_p)=2\floor{\frac{p+2\ell+1}{2(2\ell+1)}}.
\]
The proof combines a sharp additive zero-sum odd-girth argument with weak odd pancyclicity to transfer the result from odd-girth exclusion to fixed odd-cycle exclusion.  We also give a canonical extremal family, an exact extremality criterion in terms of odd zero-sum avoidance, and an example showing that extremizers need not be dilates of the canonical construction.
\end{abstract}

\maketitle

\section{Introduction}
Given a graph $F$, a host graph $G$ is \emph{$F$-free} if it does not contain $F$ as a subgraph.
The \emph{Tur\'{a}n number} $\mathrm{ex}(n,F)$ of $F$ is the maximum number of edges in an $F$-free graph on $n$ vertices. The study of $\mathrm{ex}(n,F)$ is a central topic in extremal graph theory~\cite{Bollobas1998,Simonovits1968}.  One of the classical results in the area is Tur\'{a}n's theorem~\cite{Mantel1907,Turan1941},
which states that for every integer $\ell \ge 2$ the Tur\'{a}n number $\mathrm{ex}(n,K_{\ell+1})$
is uniquely attained by the complete balanced $\ell$-partite graph on $n$ vertices,
called the Tur\'{a}n graph.  The Erd\H{o}s--Stone theorem~\cite{ErdosStone1946} gives the asymptotic value of $\mathrm{ex}(n,F)$ in terms of the chromatic number of $F$.

Classical Tur\'{a}n-type problems maximize the number of edges subject to forbidden-subgraph constraints.
From an algebraic point of view, however, such problems are generally insensitive to the symmetries that define group-derived graph classes: the parts of a Tur\'{a}n graph are chosen as a set partition of the vertex set, and complete adjacency between two parts need not encode a group law or a translation-invariant difference rule.  Cayley graphs retain this algebraic structure~\cite{Biggs1993,GodsilRoyle2001}.  Once the vertex set is identified with a group, all edges are generated uniformly by a fixed connection set.  In this paper we study a Cayley--Tur\'{a}n problem in which the admissible host graphs are constrained by this translation-invariant symmetry.

Let $p$ be an odd prime, let $\Z_p$ denote the residue classes modulo $p$, and let $S\subseteq \Z_p$ be symmetric, that is, $S=-S$, with $0\notin S$.  We write $\Cay(\Z_p,S)$ for the \emph{undirected Cayley graph} with vertex set $\Z_p$, in which two vertices $x,y\in\Z_p$ are adjacent if and only if $x-y\in S$. The set $S$ is the \emph{connection set}.  For a graph $F$, define the \emph{Cayley--Tur\'{a}n number} over $\Z_p$ by
\[
\ex_{\Cay}(F,\Z_p)=
\max\{|S|:S=-S,\ 0\notin S,\ \Cay(\Z_p,S)\text{ contains no copy of }F\}.
\]
Determining the Cayley--Tur\'{a}n number of a graph is not a direct translation of the ordinary Tur\'{a}n problem: a forbidden subgraph becomes a system of additive equations together with non-degeneracy conditions in $S$~\cite{Bajnok2017,TaoVu2006}.

Cashman and Kelley \cite{CashmanKelley2025} initiated this line of study.  They studied Cayley--Tur\'{a}n numbers of cycles and proved that $\ex_{\Cay}(C_{2\ell},\Z_p)=2$ for $2\le \ell<p/2$. For odd cycles, they conjectured the following formula.
\begin{conjecture}\label{conjecture-odd-cycle}
For $p>2\ell+1$, we have 
\[
\ex_{\Cay}(C_{2\ell+1},\Z_p)
=2\floor{\frac{p+2\ell+1}{2(2\ell+1)}}.
\]    
\end{conjecture}
They confirmed the case $\ell=1$. In this paper, we prove \Cref{conjecture-odd-cycle}.

\begin{theorem}\label{thm:main}
Let $p$ be an odd prime and let $\ell\ge 1$ be an integer with $p\ge 2\ell+1$.
If $p>2\ell+1$, then 
\[
\ex_{\Cay}(C_{2\ell+1},\Z_p)=2\floor{\frac{p+2\ell+1}{2(2\ell+1)}}.
\]
Otherwise, $\ex_{\Cay}(C_{2\ell+1},\Z_{2\ell+1})=0$. 
\end{theorem}

The proof combines two ingredients: an additive odd-girth extremal theorem based on the Cauchy--Davenport theorem \cite{Cauchy1813,Davenport1935}, and a weak odd pancyclicity theorem of Alspach, Bendit, and Maitland for abelian Cayley graphs \cite{AlspachBenditMaitland2013}.  These ingredients also yield sharp extremal examples and a concise criterion for extremality.

\subsection*{Organization}
\Cref{sec:prelim} records notation and the zero-sum correspondence.  \Cref{sec:girth} proves the additive odd-girth theorem.  \Cref{sec:fixed} proves the main theorem.  \Cref{sec:extremal} gives canonical extremal constructions and an exact extremality criterion.  \Cref{sec:remarks} discusses the scope of the method and possible extensions.

\section{Preliminaries}\label{sec:prelim}

This section fixes the additive notation used throughout the paper and records the standard tools needed later.  
Throughout, $p$ is an odd prime.
For $A_1,\ldots,A_k\subseteq\Z_p$, let 
\[
A_1+\cdots+A_k=\{a_1+\cdots+a_k:a_i\in A_i\}.
\]
For $A\subseteq\Z_p$, let 
\[
kA=\underbrace{A+\cdots+A}_{k\text{ times}},
\]
where repeated summands are allowed.  Thus $0\in kA$ means that there exist, not necessarily distinct, elements $a_1,\ldots,a_k\in A$ such that $a_1+\cdots+a_k=0$. We use the following standard form of the Cauchy--Davenport Theorem .

\begin{theorem}[Cauchy--Davenport Theorem]\label{thm:CD}\cite{Cauchy1813,Davenport1935}
If $A_1,\ldots,A_k$ are nonempty subsets of $\Z_p$, then
\[
|A_1+\cdots+A_k|\ge \min\{p, |A_1|+\cdots+|A_k|-k+1\}.
\]
In particular,
\[
|kA|\ge \min\{p,k|A|-(k-1)\}.
\]
\end{theorem}

Recall that a \emph{closed walk of length $j$} is a vertex sequence $v_0v_1\cdots v_j$ in a graph satisfying $v_j=v_0$ and with every consecutive pair adjacent. Every closed walk can be decomposed into an edge-disjoint union of cycles. We shall use the following elementary observation.

\begin{observation}\label{lem:closed-walk-cycle}
If a simple graph contains a closed walk of odd length $j$, then it contains an odd cycle of length at most $j$.    
\end{observation}

The next lemma is the elementary bridge between additive relations and walks in Cayley graphs.  

\begin{lemma}\label{lem:walk-zero}
Let $S\subseteq\Z_p$ be symmetric with $0\notin S$, and let $j\ge1$.
\begin{enumerate}[label=(\roman*)]
\item If $s_1,\ldots,s_j\in S$ and $s_1+\cdots+s_j=0$, then $\Cay(\Z_p,S)$ contains a closed walk of length $j$.
\item If, in addition, the partial sums
\[
0,\quad s_1,\quad s_1+s_2,\quad\ldots,\quad s_1+\cdots+s_{j-1}
\]
are pairwise distinct, then $\Cay(\Z_p,S)$ contains a cycle of length $j$.
\item Conversely, every cycle of length $j$ in $\Cay(\Z_p,S)$ gives elements $s_1,\ldots,s_j\in S$ with $s_1+\cdots+s_j=0$ and pairwise distinct partial sums as above.
\end{enumerate}
\end{lemma}

\begin{proof}
For (i), define vertices
\[
v_0=0,\qquad v_i=s_1+\cdots+s_i\quad (1\le i\le j).
\]
Then $v_j=s_1+\cdots+s_j=0=v_0$, so the sequence returns to its starting point.  For each $i$, it follows from 
\[
v_i-v_{i-1}=s_i\in S
\]
that  $v_{i-1}$ and $v_i$ are adjacent in $\Cay(\Z_p,S)$ for every $i$. Thus, $v_0v_1\cdots v_j$ is a closed walk of length $j$.

Part (ii) follows because the same vertices as in (i) form a cycle under the stated distinctness assumption.

For (iii), let $v_0v_1\cdots v_{j-1}v_j$ be a cycle and set $v_j=v_0$.  For $1\le i\le j$, define $s_i=v_i-v_{i-1}$. 
Then 
\[
\sum_{i=1}^j s_i=v_j-v_0=0.
\]
Moreover,
\[
s_1+\cdots+s_i=v_i-v_0\qquad (1\le i\le j).
\]
After translating the cycle by $-v_0$, its vertices become
\[
0,\quad s_1,\quad s_1+s_2,\quad\ldots,\quad s_1+\cdots+s_{j-1}.
\]
This completes the proof. 
\end{proof}

\section{The additive theorem for odd girth}\label{sec:girth}

This section proves the additive core of the paper: an exact bound for symmetric zero-free connection sets whose Cayley graphs have odd girth greater than a prescribed odd integer.  For a graph $G$, the \emph{odd girth} of $G$, denoted by $\og(G)$, is the length of its shortest odd cycle.  If $G$ has no odd cycle, equivalently, if $G$ is bipartite, set $\og(G)=\infty$. For an odd integer $m\ge 3$, define
\[
M_m(p)=\max\{|S|: S=-S,\ 0\notin S,\ \og(\Cay(\Z_p,S))>m\}.
\]

\begin{theorem}\label{thm:odd-girth}
Let $p$ be an odd prime and let $m\ge3$ be odd.  Then
\[
M_m(p)=2\floor{\frac{p+m-2}{2m}}.
\]
\end{theorem}

\begin{proof}
Set
\[
r=\floor{\frac{p+m-2}{2m}}.
\]
We prove the upper and lower bounds separately.

\smallskip
\noindent\textbf{Upper bound.}
Let $S\subseteq\Z_p$ satisfy $S=-S$, $0\notin S$, and $\og(\Cay(\Z_p,S))>m$.  If $S=\emptyset$, there is nothing to prove.  Assume $S\ne\emptyset$.

Since $\og(\Cay(\Z_p,S))>m$, there is no odd cycle of length at most $m$.  We first claim that $0\notin mS$. If not, then there exist $s_1,\ldots,s_m\in S$ with
\[
s_1+\cdots+s_m=0.
\]
By \Cref{lem:walk-zero}, this gives a closed walk of length $m$, which together with \Cref{lem:closed-walk-cycle} gives an odd cycle of length at most $m$, contradicting $\og(\Cay(\Z_p,S))>m$.  Therefore $0\notin mS$.  In particular $mS$ is a proper subset of $\Z_p$, so 
\begin{align}\label{bound-mS-p-1}
|mS|\le p-1.     
\end{align}

On the other hand, by \Cref{thm:CD}, 
\[
|mS|\ge \min\{p,m|S|-(m-1)\}.
\]
Together with \eqref{bound-mS-p-1}, this yields
\[
m|S|-(m-1)\le |mS|\le p-1,
\]
and so
\[
|S|\le \frac{p+m-2}{m}.
\]
Recall that $S=-S$ and $0\notin S$, which implies that $|S|$ is even.
Therefore, $|S|\le2r$.

\smallskip
\noindent\textbf{Lower bound.}
If $r=0$, take $S_r=\emptyset$.  Then $\og(\Cay(\Z_p,S_r))=\infty>m$, as required.  Hence we may assume that $r\ge1$. Let
\[
A=\frac{p+1}{2}-r,
\qquad
B=\frac{p-1}{2}+r.
\]
It follows from $m\ge3$ that $r\le (p-1)/2$, so $1\le A\le B\le p-1$.

Let 
\[
S_r=\{A,A+1,\ldots,B\}\pmod p.
\]
Then $0\notin S_r$ and $|S_r|=B-A+1=2r$. To check symmetry, write an element of $S_r$ in the form
\[
x=\frac{p+1}{2}+t,
\qquad -r\le t\le r-1.
\]
Then
\[
-x\equiv \frac{p-1}{2}-t=\frac{p+1}{2}+(-t-1)\pmod p,
\]
and $-t-1\in[-r,r-1]$.  Thus $S_r=-S_r$.

It remains to show that $\Cay(\Z_p,S_r)$ has no odd cycle of length at most $m$.  By the converse part of \Cref{lem:walk-zero}, any odd cycle of length $j$ would give $0\in jS_r$.  Therefore it suffices to show that for every odd integer $j$ with $1\le j\le m$,
\[
0\notin jS_r.
\]
Let $j=2q+1\le m$, and suppose $s_1,\ldots,s_j\in S_r$.  Represent each $s_i$ by its integer representative in $[A,B]$ and set
\[
T=s_1+\cdots+s_j.
\]
Then $jA\le T\le jB$.  A direct calculation gives
\[
jA-qp=\frac{p+j-2jr}{2}
\]
and
\[
(q+1)p-jB=\frac{p+j-2jr}{2}.
\]
Since $r\le (p+m-2)/(2m)$ and $j\le m$,
\[
2jr\le \frac{j(p+m-2)}{m}.
\]
Consequently,
\[
p+j-2jr\ge p+j-\frac{j(p+m-2)}{m}=\frac{(m-j)p+2j}{m}>0.
\]
Thus $jA-qp>0$ and $(q+1)p-jB>0$.  Since $jA\le T\le jB$, we obtain the strict inequalities
\[
qp<T<(q+1)p.
\]
There is no multiple of $p$ strictly between the consecutive multiples $qp$ and $(q+1)p$.  Hence $T$ is not divisible by $p$, so
\[
s_1+\cdots+s_j\not\equiv0\pmod p.
\]
This proves $0\notin jS_r$ for every odd $j\le m$.  This completes the proof.
\end{proof}

\section{From odd girth to fixed odd cycles}\label{sec:fixed}
In this seciton, we prove \Cref{thm:main}. The only external graph-theoretic input is the following form of the weak odd pancyclicity theorem of Alspach, Bendit, and Maitland \cite{AlspachBenditMaitland2013}.

\begin{theorem}[Alspach--Bendit--Maitland \cite{AlspachBenditMaitland2013}]\label{thm:ABM}
Let $G$ be a finite simple connected Cayley graph of valency at least $3$ on an abelian group.  If $G$ is not bipartite, then $G$ contains cycles of every odd length $L$ satisfying
\[
\og(G)\le L\le |V(G)|.
\]
\end{theorem}

\begin{lemma}\label{lem:degree-two}
If $S\subseteq\Z_p$ satisfies $S=-S$, $0\notin S$, and $|S|=2$, then $\Cay(\Z_p,S)$ is a cycle of length $p$.
\end{lemma}

\begin{proof}
Write $S=\{a,-a\}$ with $a\ne0$.  Since $p$ is prime, every nonzero element of $\Z_p$ generates the additive group, so
\[
0,a,2a,\ldots,(p-1)a
\]
are all the elements of $\Z_p$.  The edges of $\Cay(\Z_p,S)$ are exactly the pairs $x\sim x+a$ and $x\sim x-a$, which form a cycle of length $p$.
\end{proof}

\begin{proof}[Proof of \Cref{thm:main}]
Put $m=2\ell+1$.  Then $m$ is odd with $3\le m\le p$.  For simplicity, set 
\[
R=2\floor{\frac{p+m-2}{2m}}.
\]
It suffices to prove that $\ex_{\Cay}(C_{2\ell+1},\Z_p)=R$.

\smallskip
\noindent\textbf{Lower bound.}
Let $S_r$ be the set constructed in the lower-bound part of \Cref{thm:odd-girth}, with
\[
r=\floor{\frac{p+m-2}{2m}}. 
\]
Then $\og(\Cay(\Z_p,S_r))>m$.  Therefore
\[
\ex_{\Cay}(C_{2\ell+1},\Z_p)\ge R.
\]

\smallskip
\noindent\textbf{Upper bound.}
Let $S\subseteq\Z_p$ satisfy $S=-S$, $0\notin S$, and suppose $\Cay(\Z_p,S)$ contains no $C_{2\ell+1}$.  We prove $|S|\le R$. We may assume that $|S|\ge 2$, since otherwise $S=\emptyset$ and there is nothing to prove.

If $|S|=2$, then by \Cref{lem:degree-two}, $\Cay(\Z_p,S)\cong C_p$.  If $p=2\ell+1$, this graph is itself a copy of $C_{2\ell+1}$, a contradiction. Otherwise, $R\ge2=|S|$, and the desired bound follows.

It remains to consider $|S|\ge4$.  We check the hypotheses of \Cref{thm:ABM}.  Let $G=\Cay(\Z_p,S)$.  Since $S$ is nonempty, choose $a\in S$. Then $a\ne0$, and the subgroup generated by $a$ is all of $\Z_p$.  The edges using the two generators $\pm a$ already connect all vertices, so $G$ is connected.  Moreover, its valency is $|S|\ge4$, because from each vertex $x$ the neighbors are exactly $x-s$ with $s\in S$.  Finally, $G$ is not bipartite.  Indeed, if a $d$-regular bipartite graph with $d>0$ has bipartition $U,V$, then counting edges from the two sides gives $d|U|=d|V|$, so $|U|=|V|$.  This is impossible here because $|V(G)|=p$ is odd.  Thus all hypotheses of \Cref{thm:ABM} are satisfied.

Let $g=\og(G)$.  By \Cref{thm:ABM}, $G$ contains cycles of every odd length $L$ with $g\le L\le p$.  Since $m=2\ell+1$ is odd and $m\le p$, if $g\le m$, then \Cref{thm:ABM} gives a cycle of length $m$, contradicting the assumption that $G$ is $C_{2\ell+1}$-free.  Hence $g>m$.  Applying \Cref{thm:odd-girth} gives
\[
|S|\le M_m(p)=R.
\]
This proves the upper bound.

Finally, it suffices to show that if $p>2\ell+1$, then
\[
\floor{\frac{p+m-2}{2m}}=\floor{\frac{p+m}{2m}}.
\]
Indeed, the two real numbers differ by $1/m<1$.  Therefore their floors can differ only if the interval
\[
\left[\frac{p+m-2}{2m},\frac{p+m}{2m}\right]
\]
contains an integer.  Equivalently, the residue of $p+m-2$ modulo $2m$ must be $2m-2$ or $2m-1$.  These two possibilities give respectively
\[
p\equiv m\pmod{2m}
\qquad\text{or}\qquad
p\equiv m+1\pmod{2m}.
\]
The second is impossible because $p$ is odd while $m+1$ is even.  The first would imply $m\mid p$, impossible when $p$ is prime and $p>m$ with $m\ge3$.  This proves the equivalent formula for $p>2\ell+1$.
\end{proof}

\section{Extremal constructions}\label{sec:extremal}

This section records the extremal examples and the corresponding optimality criterion.  Fix an odd prime $p$ and an integer $\ell$ with $1\le \ell\le (p-1)/2$.  Put $m=2\ell+1$, and write
\[
r=\floor{\frac{p+m-2}{2m}},\qquad R=2r.
\]
Define the central interval
\[
I_{p,m}=\begin{cases}
\emptyset, & r=0,\\[2mm]
\left\{\dfrac{p+1}{2}-r,\dfrac{p+1}{2}-r+1,\ldots,\dfrac{p-1}{2}+r\right\}\pmod p, & r\ge1.
\end{cases}
\]
For $a\in\Z_p^\times$, let
\[
I_{p,m}(a)=aI_{p,m}:=\{ax:x\in I_{p,m}\}.
\]
Here multiplication by $a$ is taken in the field $\Z_p$; since $a\ne0$, it is a bijection of $\Z_p$.

\begin{proposition}\label{prop:canonical}
For every $a\in\Z_p^\times$, the set $I_{p,m}(a)$ satisfies
\[
I_{p,m}(a)=-I_{p,m}(a),\qquad 0\notin I_{p,m}(a),\qquad |I_{p,m}(a)|=R,
\]
and the Cayley graph generated by $I_{p,m}(a)$ contains no $C_{2\ell+1}$.  Consequently, it attains the Cayley--Tur\'{a}n number for $C_{2\ell+1}$ over $\Z_p$.
\end{proposition}

\begin{proof}
The case $a=1$ is exactly the lower-bound construction in the proof of \Cref{thm:odd-girth}.  In particular, $I_{p,m}$ is symmetric, zero-free, has size $R$, and satisfies
\[
0\notin jI_{p,m}
\]
for every odd integer $j$ with $1\le j\le m$.

Multiplication by $a$ is an automorphism of the additive group $\Z_p$.  Hence it preserves cardinality, fixes $0$, and commutes with taking negatives:
\[
a(-I_{p,m})=-(aI_{p,m}).
\]
It also preserves the zero-sum avoidance condition.  Indeed,
\[
j(aI_{p,m})=a(jI_{p,m}),
\]
so $0\in j(aI_{p,m})$ would imply $0\in jI_{p,m}$ after multiplying by $a^{-1}$.  Thus the same properties hold for $aI_{p,m}$.  By \Cref{lem:walk-zero} and \Cref{lem:closed-walk-cycle}, $\Cay(\Z_p,I_{p,m}(a))$ has odd girth greater than $m=2\ell+1$, and in particular contains no $C_{2\ell+1}$.  Since \Cref{thm:main} gives the upper bound $R$, this graph is extremal.
\end{proof}

Multiplication by a nonzero scalar also shows that all graphs in the canonical family are mutually isomorphic as Cayley graphs.  Indeed, the map $x\mapsto a^{-1}x$ sends $\Cay(\Z_p,aI_{p,m})$ to $\Cay(\Z_p,I_{p,m})$. The following criterion gives a convenient form of the full extremality condition.

\begin{proposition}\label{prop:criterion}
Let $S\subseteq\Z_p$ satisfy $S=-S$, $0\notin S$, and $|S|=R$.  Then the following are equivalent:
\begin{enumerate}[label=(\roman*)]
\item $\Cay(\Z_p,S)$ is $C_{2\ell+1}$-free;
\item $\og(\Cay(\Z_p,S))>2\ell+1$;
\item $0\notin jS$ for every odd integer $j$ with $1\le j\le 2\ell+1$.
\end{enumerate}
Consequently, the extremal graphs are exactly the Cayley graphs generated by symmetric zero-free sets $S$ of size $R$ satisfying the odd zero-sum avoidance condition in (iii).
\end{proposition}

\begin{proof}
The implication (ii)$\Rightarrow$(i) is immediate.

We next prove (ii)$\Leftrightarrow$(iii).  Suppose first that (iii) fails.  Then for some odd $j\le 2\ell+1$ there are elements $s_1,\ldots,s_j\in S$ with $s_1+\cdots+s_j=0$.  By \Cref{lem:walk-zero}, this gives a closed walk of length $j$, and by \Cref{lem:closed-walk-cycle} it contains an odd cycle of length at most $j\le 2\ell+1$.  Hence $\og(\Cay(\Z_p,S))\le 2\ell+1$, so (ii) fails.  Conversely, if (ii) fails, then there is an odd cycle of some length $j\le 2\ell+1$.  The converse part of \Cref{lem:walk-zero} gives elements of $S$ whose sum is $0$, so $0\in jS$ and (iii) fails.

It remains to prove (i)$\Rightarrow$(ii), using the additional assumption $|S|=R$.  If $R=0$, then $S=\emptyset$ and there are no cycles, so (ii) is clear.  If $R=2$, then $S=\{a,-a\}$ for some nonzero $a$, and \Cref{lem:degree-two} gives $\Cay(\Z_p,S)\cong C_p$.  Since $R=2$ cannot occur when $2\ell+1=p$, we have $2\ell+1<p$ in this case; hence the only odd cycle in this degree-two Cayley graph has length $p>2\ell+1$, and $\og(\Cay(\Z_p,S))=p>2\ell+1$.

Finally suppose $R\ge4$.  As in the proof of \Cref{thm:main}, the graph $\Cay(\Z_p,S)$ is a finite simple connected non-bipartite Cayley graph on the abelian group $\Z_p$, and its valency is at least $4$.  If its odd girth were at most $2\ell+1$, then \Cref{thm:ABM} would give cycles of all odd lengths from the odd girth up to $p$, including a cycle of length $2\ell+1$.  This contradicts (i).  Thus (ii) holds.
\end{proof}

\begin{remark}
The canonical interval family does not give uniqueness.  For example, let $p=23$ and $\ell=2$, so $2\ell+1=5$.  Then $R=4$, and
\[
S=\{\pm1,\pm5\}\subseteq\Z_{23}
\]
is extremal.  Indeed, write any sum of $j\in\{3,5\}$ elements of $S$ as $u+5v$, where $u,v\in\Z$, $|u|+|v|\le j$, and $u+v\equiv j\pmod2$.  If $j=3$, then $|u+5v|\le15<23$, so a zero congruence modulo $23$ would force $u+5v=0$.  This implies $u=-5v$ and hence $6|v|\le3$, so $u=v=0$, contradicting $u+v\equiv1\pmod2$.  If $j=5$, then $u+5v\in[-25,25]$, so a zero congruence modulo $23$ would require $u+5v\in\{0,\pm23\}$.  The case $0$ is excluded as above, and for $t\in\{\pm23\}$ one checks that
\[
|t-5v|+|v|>5\qquad(-5\le v\le5),
\]
so $u+5v=t$ is impossible under $|u|+|v|\le5$.  Therefore $0\notin3S\cup5S$, and \Cref{prop:criterion} shows that $\Cay(\Z_{23},S)$ is extremal.

This $S$ is not a multiplicative dilation of the central interval $I_{23,5}=\{\pm10,\pm11\}$.  For a four-point symmetric set $\{\pm u,\pm v\}$, the unordered ratio set $\{\pm v/u,\pm u/v\}$ is invariant under dilation.  For $\{\pm10,\pm11\}$ this ratio set is $\{\pm8,\pm3\}$, whereas for $\{\pm1,\pm5\}$ it is $\{\pm5,\pm14\}$ in $\Z_{23}$.  These sets are distinct.
\end{remark}

\section{Concluding remarks}\label{sec:remarks}

For a fixed $\ell$, \Cref{thm:odd-girth} shows that if a symmetric connection set has density exceeding
\[
\frac1{2\ell+1}+O_\ell\!\left(\frac1p\right),
\]
then $0\in (2\ell+1)S$, forcing an odd closed walk of length $2\ell+1$ and hence an odd cycle of length at most $2\ell+1$.  The interval construction centered at $p/2$ shows that this threshold is sharp for excluding all odd cycles up to length $2\ell+1$.

The passage from odd-girth exclusion to fixed $C_{2\ell+1}$-exclusion is not a purely additive statement; it uses a pancyclicity theorem for abelian Cayley graphs.  This explains why the same zero-sum argument alone proves the exact odd-girth theorem but not, by itself, the fixed-cycle theorem.

Several refinements remain natural.  The exact structural classification of all extremal connection sets remains open even in apparently small cases.  A second direction is to study cyclic groups of composite order, where the Cauchy--Davenport theorem must be replaced by Kneser-type estimates and where nontrivial subgroups introduce additional extremal phenomena.

\end{document}